\documentclass[11pt, reqno]{amsart}

\usepackage{amssymb,latexsym}
\usepackage{graphicx}
\usepackage{url}		
\usepackage{fancyhdr}
\calclayout

\theoremstyle{theorem}
\newtheorem{theorem}{Theorem}
\theoremstyle{conjecture}
\newtheorem{conjecture}{Conjecture}
\theoremstyle{definition}

\theoremstyle{lemma}
\newtheorem{lemma}{Lemma}
\theoremstyle{example}

\newtheorem*{main result}{Main result}
\begin{document}
\title[M Diouf]{A conditional proof of Legendre's  Conjecture \\and Andrica's conjecture}
\maketitle

\begin{center}
{\bf Madieyna Diouf}\\
\small {\email{mdiouf1@asu.edu}}
\end{center}

\begin{abstract}
The Legendre conjecture has resisted analysis over a century, even under assumption of the  Riemann Hypothesis. We present, a significant improvement on previous results by greatly reducing the assumption to a more modest statement called the Parity conjecture. 
Let $p_n$ and $p_{n+1}$ be two consecutive odd primes, let $m$ be their midpoint fixed once for all. 
\begin{conjecture}
The largest multiple of $p_n$ not exceeding ${m_i}^2$ is odd for every integer $m_i$ in the interval $(p_n, m]$. 
\end{conjecture}
\begin{main result} \textup{We prove that the Parity conjecture  implies Legendre's conjecture and Andrica's conjecture. }
\end{main result}
\end{abstract}

\section{Introduction And Statement of Results}
The study of maximal gaps between consecutive primes is an important subject that is actively pursued and the Bertrand's postulate \cite{Bertrand} is one of its first consequences. In $1850$, Chebyshev proved the Bertrand’s postulate \cite{Tchebychev}, and  P. Erd\" os presented a simplified proof in $1932$  \cite{Erdos}. Strong results were also obtained in the generalizations of Bertrand's Postulate. In  $2006$, M. El Bachraoui proved the existence of a prime in the interval $[2n, 3n]$ \cite{Bachraoui}.  In $2011$, Andy Loo exhibited a proof that shows not only the existence of a prime between $3n$ and $4n$, but also the infinitude of the number of primes in this interval when $n$ goes to infinity  \cite{Andy}. Pierre Dusart gave the best known result in this category when he improved in $2016$ his previous work by showing that there is a prime between $x$ and $(x+x/(25{\log}^2 x))$ for $x\geq468991632$ \cite{Pierre}.\\

\lq\lq On $25th$ October $1920$ G. H. Hardy read Cram\'er's paper \textit{\lq\lq On the distribution of primes"} to the Cambridge Philosophical Society. Here Cram\'er develops a statistical approach to this question showing that for any $\epsilon>0$
\begin{center}
$ p_{n+1}-p_n=O(p^{\epsilon}$)
\end{center}
for \lq most\rq \: $p_n$: in fact for all but at most $x^{1-3\epsilon/2}$ of the primes $p_n\leq x$.\rq\rq \cite{Andrew}. As a result of the Prime Number Theorem alone, we have
$p_{n+1}-p_n<\epsilon p_n$ for all $\epsilon>0$.\\
By the  Prime Number Theorem with error term, we obtain 
\begin{center}
$p_{n+1}-p_n<\frac{p_n}{{(\log p_n)^c}}$ \textup{\:  for some positive constant $c$}.
\end{center}

Based on observations that revolve around the midpoint $m$ of two consecutive odd primes $p_n$ and $p_{n+1}$, the largest multiple of $p_n$ not exceeding $m^2$, the  Bertrand's postulate and few other properties, we show that the gap $g_n$ between two consecutive primes satisfies $g_n=O({p_n}^{1/2})$. It is indeed shown precisely that 
\begin{center}
$p_{n+1}-p_{n} <2{p_n}^{1/2}.$
\end{center}

Close results are obtained under the assumption of the Riemann Hypothesis. Harald Cram\'er proved that if the Riemann hypothesis holds, then the gap $g_n$ satisfies $g_n=O(\sqrt{p_n}\log p_n)$ \cite{Andrew}.  The best unconditional bound is known to Baker, Harman and Pintz, who proved the existence of $x_0$ such that there is a prime in the interval $[x, x+O(x^{21/40})]$ for $x>x_0$, \cite{Baker}.\\ 

Significant works have been done on the upper bound of the gap between consecutive primes by various authors without assuming an unproved hypothesis. Hoheisel was the first to show in $1930$ the existence of a constant $\delta>0$ (mainly $\delta = 1/33000$) such that $p_{n+1}-p_n = O({p_n}^{1-\delta})$ \cite{Hoheisel}. Heilbronn \cite{Heibronn}, and Tchudakoff \cite{Tchuda}, both improved on the value of $\delta$. Ingham \cite{Ingham} made a significant progress that contributed to the first solutions surrounding the problem of existence of a prime between two consecutive cubes.\\

The key ideas in the proof that allow us to obtain unconditionally our result, are the principal of induction and the Parity conjecture.\\

\begin{lemma}
If $a$ and $c$ are two positive integers and $m_i$ is their midpoint, then $ac<{m_i}^2$.
\end{lemma}
\begin{proof}
Suppose that $a$ and $c$ are two positive integers and $m_i$ is their midpoint, then there exists a positive integer $b$ such that: $m_i-b=a$ and $m_i+b=c.$
\begin{align*}
(m_i-b)(m_i+b)&=ac.\\
{m_i}^2-b^2&=ac.\\
{m_i}^2-ac&=b^2.\\
{m_i}^2-ac&>0.\\
ac&<{m_i}^2.
\end{align*}
\end{proof}
\begin{lemma}
If $p_n$ and $p_{n+1}$ be two consecutive odd primes and $m$ is their midpoint, then any two consecutive squares less than or equal to $m^2$, have a gap less than $3p_n$.
\end{lemma}
\begin{proof}
Suppose that $p_n$ and $p_{n+1}$ are two consecutive odd primes and $m$ is their midpoint. We have $m^2-(m-1)^2=2m-1$, therefore,  the distance between any two consecutive squares less than or equal to $m^2$, is less than $2m-1$. It suffices now to show that $2m-1<3p_n$. 
\begin{align*}
2m-1 &\leq p_{n+1}+p_n-1.\\
\textup{By \cite{Bertrand}, } p_{n+1} &<2p_n.\\
\textup{Thus, } \: 2m-1 &<2p_n+p_n-1.\\
2m-1 &<3p_n.
\end{align*}
\end{proof}
\textbullet \: Let $L(p_n, {m_i}^2)$ denote the largest multiple of $p_n$ not exceeding ${m_i}^2$. 
\begin{lemma}
If $p_n$ and $p_{n+1}$ are two consecutive odd primes and $m$ to be their midpoint, then \\$p_n(2m_i-p_n)$ is the largest multiple of $p_n$ not exceeding ${m_i}^2$ for every integer $m_i$ in the interval $(p_n, m]$. That is, 
\begin{center}
$L(p_n, {m_i}^2)=p_n(2m_i-p_n)$ for all integers $m_i$ in the interval $(p_n, m]$.
\end{center}
\end{lemma}
\begin{proof}
Suppose that $p_n$ and $p_{n+1}$ are two consecutive odd primes and $m$ is their midpoint. \\
We proceed by induction on $m_i$ in the interval $(p_n, m)$.\\
\textbf{Base case:} If $m_i=p_n+1$, then 
\begin{align}
p_n(2m_i-p_n) &= p_n(2(p_{n}+1)-p_n).\\
&=p_n(p_n+2).\\
{p_n}^2&<p_n(p_n+2)<(p_n+1)^2.
\end{align}
There is no multiple of $p_n$ greater than $p_n(p_n+2)$ between $p^2$ and $(p_n+1)^2$. Therefore, 
\begin{align}
L(p_n, {(p_n+1)}^2)&=p_n(p_n+2).\\
L(p_n, {(p_n+1)}^2)&= p_n(2(p_{n}+1)-p_n).
\end{align}
\textbf{Inductive hypothesis:} Suppose that for some integer $m_i$ in the interval $(p_n, m)$,  
\begin{align}
 L(p_n, {m_i}^2) &=p_n(2m_i-p_n). 
\end{align}
\textbf{Inductive step:} The objective is to show that 
\begin{align}
 L(p_n, {(m_i+1)}^2)&=p_n(2(m_i+1)-p_n).
\end{align}
The inductive hypothesis $(6)$, implies that 
\begin{align}
 p_n(2m_i-p_n)+p_n &> L(p_n, {m_i}^2).\\
 p_n(2m_i-p_n) +p_n &> {m_i}^2.\\
{m_i}^2&< p_n(2m_i-p_n) +p_n .
\end{align} 
By Lemma $1$, 
\begin{align}
 p_n(2(m_i+1)-p_n)&<(m_i+1)^2 .
\end{align} 
$(10)$ and $(11)$ give
\begin{align}
{m_i}^2&< p_n(2m_i-p_n) +p_n< p_n(2(m_i+1)-p_n)<(m_i+1)^2 .
\end{align}
It is clear that $p_n$ is the largest prime less than $m_i+1$ for all $m_i$ in the interval $(p_n, m)$.  Applying the Parity conjecture on $m_i+1$ implies that \\
\textbf{a)} $L(p_n, {(m_i+1)}^2)$ is odd. \\
\textbf{b)} $p_n(2(m_i+1)-p_n)$ is the largest odd multiple of $p_n$ between ${m_i}^2$ and $(m_i+1)^2$.
\begin{center}
Justifying statement \textbf{b)}
\end{center}
 In the contrary, suppose that there is an odd multiple of $p_n$ greater than $p_n(2(m_i+1)-p_n)$ between ${m_i}^2$ and $(m_i+1)^2$. Then this odd multiple of $p_n$ would be greater than or equal to $p_n(2(m_i+1)-p_n) +2p_n$ and the difference between the last and the first multiple of $p_n$ between ${m_i}^2$ and $(m_i+1)^2$ would be, in light of $(12)$, greater than or equal to
\begin{center}
$p_n(2(m_i+1)-p_n) +2p_n-p_n(2m_i-p_n)-p_n=3p_n$. 
\end{center}
This is impossible by Lemma $2$, where it is shown that there is no gap of $3p_n$ between any two consecutive squares less than or equal to $m^2$. Therefore, there is no odd multiple of $p_n$ greater than $p_n(2(m_i+1)-p_n)$ between ${m_i}^2$ and $(m_i+1)^2$. Statement \textbf{b)} is justified.  \\
Observations \textbf{a)} and \textbf{b)} imply that
\begin{align}
 L(p_n, {(m_i+1)}^2)&=p_n(2(m_i+1)-p_n).
\end{align}
\textbf{Conclusion:} By the principle of induction, $L(p_n, {m_i}^2)=p_n(2m_i-p_n)$ for all integers $m_i$ in the interval $(p_n, m]$. That is, $p_n(2m_i-p_n)$ is the largest multiple of $p_n$ no exceeding ${m_i}^2$ for all integers $m_i$ in the interval $(p_n, m]$.\\
\end{proof}
\begin{theorem}
If $p_n$ and $p_{n+1}$ are two consecutive odd primes and $m$ to be their midpoint, then $p_{n+1}-p_{n} <2\sqrt{p_n}.$
\end{theorem}
\begin{proof}
Lemma $3$ holds for every integer $m_i$ in the interval $(p_n, m]$, in particular, for the midpoint $m$ of $p_n$ and $p_{n+1}$. Thus $p_n(2m-p_n)$ is the largest multiple of $p_n$ no exceeding ${m}^2$. Since $2m-p_n=p_{n+1}$, it follows that
\begin{center}
\textbf{a) } $p_np_{n+1}$ is the largest multiple of $p_n$ no exceeding ${m}^2$. 
\end{center}
Statement \textbf{a)} implies that 
\begin{align}
 p_np_{n+1}+p_n &> m^2.\\ 
m^2&< p_n(p_{n+1}+1).\\
(\frac{p_{n}+p_{n+1}}{2})^2 &<  p_{n}(p_{n+1}+1) .\\
{p_{n}}^2+2p_{n}p_{n+1}+{p_{n+1}}^2&< 4p_{n}p_{n+1}+4p_{n}.\\
{p_{n}}^2+2p_{n}p_{n+1}+{p_{n+1}}^2- 4p_{n}p_{n+1}&<4p_{n}.\\
(p_{n+1}-p_{n})^2 &< 4p_{n}.\\
p_{n+1}-p_{n} &<2\sqrt{p_n}.
\end{align}
\end{proof}
\section{The Legendre's Conjecture}
The conjecture states that there is a prime between $N^2$ and $(N+1)^2$ for all positive integers $N$. By $(20)$, there is prime in the interval $(p_n, p_n+2\sqrt{p_n})$. \\
Choose a positive integer $N$. Let $p_n$ be the largest prime less than $N^2$. Then 
\begin{align}
p_n <N^2<p_{n+1}&<p_n+2\sqrt{p_n}.\\
&<N^2+2N.\\
N^2<p_{n+1}&<(N+1)^2.
\end{align}
$(20)$ implies the Legendre's conjecture.
\section{Andrica's Conjecture}
Andrica's conjecture states that $\sqrt{p_{n+1}}-\sqrt{p_n}<1$ for all positive integers $n$. That is
\begin{align}
\sqrt{p_{n+1}-p_n+p_n}&<\sqrt{p_n}+1.
\end{align}
Taking the square of $(24)$ gives
\begin{align}
p_{n+1}-p_n &< 2\sqrt{p_n}+1.
\end{align}
$(20)$ implies Andrica's conjecture.
\section{Brocard's conjecture} The conjecture states the existence of at least $4$ primes in the interval $({p_n}^2, {p_{n+1}}^2)$.
\\ Unfortunately, the Parity conjecture is not strong enough to imply Brocard's conjecture. One may need to speculate strongly on distance (not just parity) in order to prove Brocard's conjecture or Oppermann's conjecture. 
\section{Oppermann's conjecture}
The conjecture \cite{Oppermann} states the following. \\For every integer $N > 1$, there is at least one prime between $N(N-1)$ and $N^2$, and at least another prime between
$N^2$ and $N(N + 1)$.\\
Same here, the Parity conjecture alone is not strong enough to wrestle down Oppermann's conjecture. 
\section {Recap: We have shown that}
If $L(p_n, {m_i}^2)$ is odd for every integer $m_i$ in the interval $(p_n, m]$ (Parity conjecture), then \\
$L(p_n, {m_i}^2) = p_n(2m_i-p_n)$ for every integer $m_i$ in the interval $(p_n, m]$ (Lemma $3$).\\
A particular case of Lemma $3$ is $m_i=m$, it implies that $L(p_n, m^2) = p_n(2m-p_n)=p_np_{n+1}$. That is statement a) in Theorem $1$ which, in view of $(14)$, can be read as, $p_n$ is an upper bound of the distance between $m^2$ and $p_np_{n+1}$.  This upper bound (that is, $p_n$), ultimately yields $(20)$ where we obtain the main result. 

\section{Discussion}
\textbullet \: The multiples of $p_n$ exhibited by the Parity conjecture are either,\\
{1)}  mixed, that is some are odd and others are even, or they are\\
{2)} all odd, or\\
{3)} all even.\\
Outcome {3)} is impossible since the base case of the induction in Lemma $3$ shows that the first of these multiples of $p_n$ is equal to $p_n(p_n+2)$ that is odd. We are left with two outcomes. The Parity conjecture is based on the second outcome, but we may ask the following questions.\\
\begin{center}
\textup{How a change of parity could affect our result?}
\textup{If the first outcome was true, how would this change the result already obtained under the assumption of the second outcome?}
\textup{Is the gap between two consecutive primes known to be in function of the parity of $L(p_n, {m_i}^2)$? \\If not, does it mean that our result also holds for the first outcome?}\\
\end{center}
\textup{\\}
\textbullet \: The Parity conjecture is in a short interval $(p_n, m]$, indeed it does not hold for all integers $m_i$ in the interval $(m, p_{n+1})$. As an example,  take the pair of consecutive primes $(23, 29)$, their midpoint is $m=26$. The Parity conjecture states that, the largest multiple of $23$ not exceeding ${m_i}^2$ is odd for all $m_i$ in the interval $(23, 26]$; this is true. Now say we go past $26$, choose $m_i =28$, it is clear that the largest multiple of $23$ not exceeding ${m_i}^2=28^2=784$, is $23*33+23=782$, that is not odd. Hence, the Parity conjecture is in a short but solid interval where it is believed to be true. 
\section{Conclusion} The Parity conjecture gives a result that is stronger than Legendre's conjecture and Andrica's conjecture, but weaker than Brocard's conjecture and Oppermann's conjecture. Nevertheless, we believe that the Parity conjecture is the easiest to prove among the $5$ conjectures,  and it is definitely much easier to prove than the Riemann hypothesis. \\
We have reduced a difficult problem, that is, the difference between two consecutive primes, to a much simpler problem where the focus is not in the gap or location of these primes, but simply in whether some particular numbers are even or odd.

\end{document}